\documentclass[a4paper,12pt]{article}

\usepackage[left=2cm,right=2cm, top=2cm,bottom=3cm,bindingoffset=0cm]{geometry}

\usepackage{verbatim}
\usepackage{amsmath}
\usepackage{amsthm}
\usepackage{amssymb}
\usepackage{delarray}
\usepackage{cite}
\usepackage{hyperref}
\usepackage{mathrsfs}
\usepackage{tikz}
\usetikzlibrary{patterns}
\usepackage{caption}
\DeclareCaptionLabelSeparator{dot}{. }
\newcommand{\be}{\beta}

\newcommand{\de}{\delta}
\newcommand{\la}{\lambda}
\newcommand{\om}{\omega}

\newcommand{\eps}{\varepsilon}

\newcommand{\iy}{\infty}

\theoremstyle{plain}

\numberwithin{equation}{section}

\newtheorem{thm}{Theorem}[section]

\newtheorem{prop}[thm]{Proposition}

\theoremstyle{definition}

\theoremstyle{remark}

\newtheorem{remark}[thm]{Remark}

\DeclareMathOperator*{\Res}{Res}

\DeclareMathOperator{\diag}{diag}

 \allowdisplaybreaks

\begin{document}

\begin{center}
{\Large\bf Spectral data asymptotics \\[0.2cm] for fourth-order boundary value problems}
\\[0.5cm]
{\bf Natalia P. Bondarenko}
\end{center}

\vspace{0.5cm}

{\bf Abstract.}  In this paper, we derive sharp asymptotics for the spectral data (eigenvalues and weight numbers) of the fourth-order linear differential equation with a distribution coefficient and three types of separated boundary conditions. Our methods rely on the recent results concerning regularization and asymptotic analysis for higher-order differential operators with distribution coefficients. The results of this paper have applications to the theory of inverse spectral problems as well as a separate significance.

\medskip


\medskip


\vspace{1cm}

\section{Introduction} \label{sec:intr}

Consider the differential equation
\begin{equation} \label{eqv}
y^{(4)} + (\tau_2(x) y')' + (\tau_1(x) y)' + \tau_1(x) y' + \tau_0(x) y = \la y(x), \quad x \in (0, 1),
\end{equation}
where $\tau_2 \in W_2^1[0,1]$ and $\tau_1 \in L_2[0,1]$ are complex-valued functions, $\tau_0$ is the generalized function of class $W_2^{-1}[0,1]$, that is, $\tau_0 = r_0'$, $r_0$ is a complex-valued function of $L_2[0,1]$,
$\la$ is the spectral parameter.
We understand equation \eqref{eqv} with a generalized function coefficient in terms of the regularization approach of Mirzoev and Shkalikov \cite{MS16} (see the details in Section~\ref{sec:prelim}).

This paper aims to derive sharp asymptotics for the eigenvalues and the weight numbers of the boundary value problems $\mathcal L_k$, $k = 1, 2, 3$, for equation \eqref{eqv} with the boundary conditions
\begin{align} \label{bc1}
    \mathcal L_1 \colon & \quad y(0) = 0, \quad y(1) = y'(1) = y''(1) = 0, \\ \label{bc2}
    \mathcal L_2 \colon & \quad y(0) = y'(0) = 0, \quad y(1) = y'(1) = 0, \\ \label{bc3}
    \mathcal L_3 \colon & \quad y(0) = y'(0) = y''(0) = 0, \quad y(1) = 0.    
\end{align}

A general effective method for obtaining eigenvalue asymptotics for arbitrary order differential operators is described in the book of Naimark \cite{Nai68}. However, derivation of sharp asymptotics for specific classes of operators requires additional efforts. In recent years, eigenvalue asymptotics for the fourth-order differential operators with various types of boundary conditions were obtained by Badanin and Korotyaev \cite{BK14, BK15}, M\"oller and Zinsou \cite{MZ17}, Aliyev et al \cite{AKM20}, Polyakov \cite{Pol22, Pol23, Pol23-2} (see also the bibliography in the mentioned papers). Motivation for investigation of spectral properties for the fourth-order operators arises from applications in mechanics, geophysics, and other fields (see, e.g., \cite{Bar74, PT01, Glad05}). Anyway, the mentioned studies deal with the case of regular (integrable) coefficients. There are much less results on spectral data asymptotics for differential equations with distribution coefficients belonging to spaces of generalized functions. Mikhailets and Molyboga \cite{MM04, MM12} investigated the eigenvalue asymptotics for the even-order differential operators $\frac{d^{2m}}{dx^{2m}} + q(x)$ with distribution potential $q \in W_2^{-m}[0,1]$ and semi-periodic boundary conditions. In \cite{Bond22-asympt}, spectral data asymptotics have been obtained for higher-order differential operators with distribution coefficients in the general form with separated boundary conditions.

In this paper, relying on the approach and on the results of \cite{Bond22-asympt}, we derive sharp asymptotics for the eigenvalues of the boundary value problems $\mathcal L_k$, $k = 1, 2, 3$. In addition, we study the asymptotic behavior of the weight numbers, which are used in the inverse spectral theory together with eigenvalues for recovering higher-order differential operators (see, e.g., \cite{Leib72, Yur02, Bond23-loc, Bond23-nsc}). We get explicit formulas for all the constants in the obtained asymptotics in terms of the functions $\tau_1$ and $\tau_2$ from equation \eqref{eqv}. Note that the boundary conditions of the problems $\mathcal L_1$ and $\mathcal L_3$ are non-self-adjoint. Therefore, it is difficult to find the starting index for numbering of their eigenvalues.
The results of this paper play an important role in the spectral data characterization for equation \eqref{eqv} (see \cite{Bond23-nsc}) and also have a separate significance.

\section{Regularization} \label{sec:prelim}

In this section, we discuss the regularization of equation \eqref{eqv} basing on the results of \cite{MS16, Bond23-mmas, Vlad17}.

Let $\sigma_0(x)$ and $\sigma_1(x)$ be the unique functions of $W_2^1[0,1]$ such that
\begin{equation} \label{defsig}
\sigma_0'' = \tau_0, \quad \sigma_0(0) = \sigma_0(1) = 0, \qquad
\sigma_1' = \tau_1, \quad \sigma_1(0) = 0.
\end{equation}

Define the matrix-function 
\begin{equation} \label{defF}
F(x) = [f_{k,j}(x)]_{k,j = 1}^4 = \begin{bmatrix}
            0 & 1 & 0 & 0 \\
            -(\sigma_1 + \sigma_0) & 0 & 1 & 0 \\
            0 & -\tau_2 + 2\sigma_0 & 0 & 1 \\
            \sigma_0^2 - \sigma_1^2 & 0 & \sigma_1 - \sigma_0 & 0
        \end{bmatrix},
\end{equation}
the quasi-derivatives
\begin{equation} \label{quasi}
y^{[0]} := y, \quad y^{[k]} := (y^{[k-1]})' - \sum_{j = 1}^k f_{k,j} y^{[j-1]}, \quad k = \overline{1,4},
\end{equation}
and the domain
$$
\mathcal D_F := \{ y \colon y^{[k]} \in AC[0,1], \, k = \overline{0,3} \}.
$$

Note that, for $y \in \mathcal D_F$, the quasi-derivative $y^{[4]}$ is correctly defined and integrable on $(0,1)$.

\begin{prop} \label{prop:reg}
For any $y \in \mathcal D_F$, the generalized function $\ell(y) := y^{(4)} + (\tau_2 y')' + (\tau_1 y)' + \tau_1 y' + \tau_0 y$ is regular and $\ell(y) = y^{[4]}$.
\end{prop}

Proposition~\ref{prop:reg} is a special case of Theorem~2.2 in \cite{Bond23-mmas}, which was obtained from the results of Vladimirov~\cite{Vlad17}. 
It follows from Proposition~\ref{prop:reg} that, for $y \in \mathcal D_F$, equation \eqref{eqv} can be equivalently represented as the first-order system
\begin{equation} \label{sys}
\vec y'(x) = (F(x) + \Lambda) \vec y(x), \quad x \in (0, 1),
\end{equation}
where 
$$
\vec y(x) = \begin{bmatrix} y(x) \\ y^{[1]}(x) \\ y^{[2]}(x) \\ y^{[3]}(x) \end{bmatrix}, \quad
\Lambda = \begin{bmatrix}
                0 & 0 & 0 & 0 \\
                0 & 0 & 0 & 0 \\
                0 & 0 & 0 & 0 \\
                \la & 0 & 0 & 0
           \end{bmatrix}.
$$
Indeed, the first three rows of the system \eqref{sys} coincide with the definitions of the quasi-derivatives \eqref{quasi}, and the last row is equivalent to the equation $y^{[4]} = \la y$. Thus, we have regularized equation \eqref{eqv} by the reduction to the system \eqref{sys} with the integrable matrix function $F(x)$.

Note that, for regularization of equation \eqref{eqv}, different constructions of an associated matrix $F(x)$ can be used (see \cite{MS16, Bond23-mmas}). Anyway, it has been proved in \cite{Bond23-reg} that the choice of an associated matrix does not influence the spectral data, which will be considered in the next sections. For the purposes of this paper, it is convenient to use the associated matrix \eqref{defF}, because it contains all the coefficients $\tau_2$, $\sigma_1$, and $\sigma_0$ of the same smoothness on the same diagonal.

\section{Birkhoff-type solutions}

In this section, we construct the Birkhoff-type solutions of the system \eqref{sys} by the well-known method (see, e.g., \cite{Tam17, Nai68}), using the technique of \cite{Yur22}. Consequently, we get the Birkhoff-type solutions for equation \eqref{eqv}.

In the asymptotic estimates below, we use the following \textbf{notations}:
\begin{itemize}
\item Put $\la = \rho^4$ and divide the complex $\rho$-plane into the sectors
$$
\Gamma_{\kappa} := \left\{ \rho \in \mathbb C \colon \frac{\pi(\kappa - 1)}{8} < \arg \rho < \frac{\pi \kappa}{8} \right\}, \quad \kappa = \overline{1,8}.
$$
\item For $\kappa \in \{ 1, 2, \dots, 8 \}$, $h > 0$, and $\rho^* > 0$, introduce the extended sector
\begin{equation*}
\Gamma_{\kappa, h, \rho^*} := \left\{ \rho \in \mathbb C \colon \rho + h \exp\bigl( \tfrac{i \pi (\kappa - 1/2)}{n}\bigr)  \in \Gamma_{\kappa}, \: |\rho| > \rho^*\right\}.
\end{equation*}
\item For a fixed sector $\Gamma_{\kappa}$, denote by $\{ \om_k \}_{k = 1}^4$ the roots of the equation $\om^4 = 1$ numbered so that
\begin{equation} \label{order}
\mbox{Re} \, (\rho \om_1) < \mbox{Re} \, (\rho \om_2) < \mbox{Re} \, (\rho \om_3) < \mbox{Re} \, (\rho \om_4), \quad \rho \in \Gamma_{\kappa}.
\end{equation}
Put $\Omega := [\om_k^{j-1}]_{j,k = 1}^4$.
\item The notation $\Upsilon(\rho)$ is used for various scalar and vector functions such that
\begin{align} \label{Ups1}
    & \Upsilon(\rho) \to 0 \quad \text{as} \quad |\rho| \to \infty, \quad \rho \in \overline{\Gamma}_{\kappa,h,\rho^*}, \\ \label{Ups2}
    & \{ \| \Upsilon(\rho_n) \| \} \in l_2 \quad \text{for any non-condensing sequence $\{ \rho_n \} \in \Gamma_{\kappa,h,\rho^*}$},
\end{align}
where $\| . \|$ stands for the vector norm if $\Upsilon(\rho)$ is a vector.
Recall that a sequence $\{ \rho_n \}_{n = 1}^{\iy}$ is called \textit{non-condensing} if 
$$
\sup_{t > 0}(N(t + 1) - N(t)) < \iy, \quad N(t) := \#\{ n \in \mathbb N \colon |\rho_n| \le t \}.
$$
\item $\Upsilon(x, \rho)$ stands for various scalar and vector functions having the properties \eqref{Ups1} and \eqref{Ups2} for each fixed $x \in [0,1]$.
\item $\eps(\rho)$ is used for various functions of form $\frac{\Upsilon(\rho)}{\rho^2}$.
\item $\{ \varkappa_n \}$ stands for various $l_2$-sequences.
\end{itemize}

The change of variables 
$$
Y(x) = \Omega^{-1} \diag \{ 1, \rho^{-1}, \rho^{-2}, \rho^{-3} \} \vec y(x)
$$
transforms the system \eqref{sys} into
\begin{equation} \label{sysL}
LY := \frac{1}{\rho} Y' - A(x, \rho) Y = 0,
\end{equation}
where
\begin{gather} \nonumber
A(x, \rho) = A_{(0)} + \frac{A_{(2)}(x)}{\rho^2} + \frac{A_{(4)}(x)}{\rho^4}, \\ \label{A1} A_{(0)} = \diag\{ \om_k \}_{k = 1}^4, \quad A_{(1)} = A_{(3)} = 0,
 \\ \label{A2} A_{(2)} := \Omega^{-1} \begin{bmatrix}
            0 & 0 & 0 & 0 \\
            -(\sigma_1 + \sigma_0) & 0 & 0 & 0 \\
            0 & -\tau_2 + 2\sigma_0 & 0 & 0 \\
            0 & 0 & \sigma_1 - \sigma_0 & 0
        \end{bmatrix} \Omega, \quad
A_{(4)} := \Omega^{-1}  \begin{bmatrix}
            0 & 0 & 0 & 0 \\
            0 & 0 & 0 & 0 \\
            0 & 0 & 0 & 0 \\
            \sigma_0^2 - \sigma_1^2 & 0 & 0 & 0
        \end{bmatrix}\Omega.
\end{gather}
Clearly, the entries of $A_{(2)}(x)$ and $A_{(4)}(x)$ belong to $W_2^1[0,1]$.

Below, we suppose that $\rho \in \overline{\Gamma}_{\kappa, h, \rho^*}$
for some $\kappa$, $h > 0$, $\rho^* > 0$, and $\{ \om_k \}_{k = 1}^4$ are numbered in the order \eqref{order} for $\rho \in \Gamma_{\kappa}$. Note that, for $\rho \in \overline{\Gamma}_{\kappa, h, \rho^*}$, we have $\mbox{Re}(\rho \om_k) - \mbox{Re} (\rho \om_{k+1}) \le c_h < \infty$. 

Let us search for approximate solutions of \eqref{sysL} in the form
\begin{equation} \label{defUk}
U_k(x,\rho) = \exp(\rho \om_k x) \left( g_{(0) k} + \frac{g_{(1)k}(x)}{\rho} + \frac{g_{(2)k}(x)}{\rho^2}\right), \quad k = \overline{1,4},
\end{equation}
where $g_{(\mu)k} = [g_{(\mu) k, j}]_{j = 1}^4$ are column vectors. Furthermore, $U_k(x, \rho)$ has to satisfy the estimate
\begin{equation} \label{estUk}
L U_k = \exp(\rho \om_k x) \rho^{-3} \left( h_k(x) + O(\rho^{-1})\right), \quad \rho \in \overline{\Gamma}_{\kappa,h,\rho^*}, \, |\rho| \to \infty, 
\end{equation}
where $h_k(x) = [h_{k,j}(x)]_{j = 1}^4$ is a column vector such that $h_{k,k} = 0$, and the $O$-estimate is uniform with respect to $x \in [0,1]$.

Due to \cite{Yur22}, the vectors $g_{(\mu)k}$, $\mu = 0, 1, 2$, and $h_k$ have to fulfill the relations
\begin{align*}
 A_{(0)} g_{(0)k} & = \om_k g_{(0)k}, \\
 A_{(0)} g_{(\mu)k} & = \om_k g_{(\mu)k} + g'_{(\mu-1)k} - \sum_{j = 1}^{\mu} A_{(j)} g_{(\mu-j)k}, \quad \mu = 1, 2, \\
 h_k & = g'_{(2)k} - \sum_{j = 1}^3 A_{(j)} g_{(3-j)k}, \quad h_{k,k} = 0.
\end{align*}

Therefore, using \eqref{A1} and \eqref{A2}, we obtain
\begin{gather} \label{gmuk}
g_{(0)k} = e_k, \quad g_{(1)k}(x) = -\frac{1}{4 \om_k} \tau_2^{\langle -1 \rangle}(x) e_k, \quad
g_{(2)k,j}(x) = \begin{cases}
                \frac{1}{(4\om_k)^2} \tau_2^{\langle -2 \rangle}(x), &  j = k, \\
                -\frac{a_{jk}(x)}{\om_j - \om_k}, & j \ne k.
            \end{cases},
\\ \nonumber
h_{k,j}(x) = - \frac{a_{jk}'(x)}{\om_j - \om_k} - \frac{1}{4 \om_k} a_{jk}(x) \tau_2^{\langle -1\rangle}(x), \quad j \ne k,
\end{gather}
where $e_k$ is the $k$-th column of the unit matrix, 
$$
\tau_2^{\langle -1 \rangle}(x) = \int_0^x \tau_2(t) \, dt, \quad
\tau_2^{\langle -2 \rangle}(x) = \int_0^x \tau_2(t) \tau_2^{\langle -1 \rangle}(t) \, dt, \quad
[a_{jk}(x)]_{j,k = 1}^4 := A_{(2)}(x).
$$
Note that $h_{k,j} \in L_2[0,1]$.

Thus, the functions $U_k(x,\rho)$, $k = \overline{1,4}$, given by \eqref{defUk} with the coefficients \eqref{gmuk} satisfy \eqref{estUk}. Let $U(x, \rho)$ be the matrix function consisting of the columns $U_k(x, \rho)$ and let $V_j(x, \rho)$, $j = \overline{1,4}$, be the rows of $V(x, \rho) := (U(x, \rho))^{-1}$. Denote
$$
L^* Z := \frac{1}{\rho} Z' - Z A(x, \rho).
$$

For $k = \overline{1,4}$, consider the integral equation
\begin{align} \nonumber
Y_k(x, \rho) = U_k(x, \rho) & + \rho \int_0^x \left( \sum_{j = 1}^k U_j(x, \rho) L^* V_j(t, \rho)\right) Y_k(t, \rho) \, dt \\ \label{inteqYk} & - \rho \int_x^1 \left( \sum_{j = k+1}^4 U_j(x, \rho) L^* V_j(t, \rho)\right) Y_k(t, \rho) \, dt,
\end{align}
whose solution $Y_k(x, \rho)$ solves the system \eqref{sysL}. Changing the variables
\begin{equation} \label{changeYk}
Y_k(x, \rho) = \exp(\rho \om_k x) W_k(x, \rho), \quad U_k(x, \rho) = \exp(\rho \om_k x) W_k^0(x, \rho)
\end{equation}
and using \eqref{estUk},
we transform equation \eqref{inteqYk} into
\begin{equation} \label{Wk}
W_k(x, \rho) = W_k^0(x, \rho) + \frac{1}{\rho^2} \int_0^1 B_k(x, t, \rho) W_k(t, \rho) \, dt,
\end{equation}
where the function $B_k(x, t, \rho)$ has the form
$$
B_k(x, t, \rho) = \begin{cases}
                    -\sum\limits_{j = 1}^k \exp(\rho(\om_j - \om_k)(x-t)) W_j^0(x, \rho) (h_{k,j}(t) + O(\rho^{-1})), \quad x \ge t, \\
                    \sum\limits_{j = k+1}^4 \exp(\rho(\om_j - \om_k)(x-t)) W_j^0(x,\rho) (h_{k,j}(t) + O(\rho^{-1})), \quad x < t,
                \end{cases}
$$
and the $O$-estimates are uniform with respect to $x \in [0,1]$ and $\rho \in \overline{\Gamma}_{\kappa,h,\rho^*}$ as $|\rho| \to \infty$. The analysis of the integral term in \eqref{Wk} shows that
\begin{equation} \label{difWk}
W_k(x, \rho) - W_k^0(x, \rho) = \frac{\Upsilon(x, \rho)}{\rho^2},
\end{equation}
where $\Upsilon(x,\rho)$ denotes various vector functions satisfying \eqref{Ups1} and \eqref{Ups2} for each fixed $x \in [0,1]$.

Combining \eqref{defUk}, \eqref{gmuk}, \eqref{changeYk}, and \eqref{difWk} all together, we obtain the fundamental solutions $\{ Y_k(x, \rho) \}_{k = 1}^4$ of the system \eqref{sysL} for $\rho \in \overline{\Gamma}_{\kappa,h,\rho^*}$:
$$
Y_k(x, \rho) = \exp(\rho \om_k x) \left( \biggl( 1 - \frac{\tau_2^{\langle -1 \rangle}(x)}{4 \om_k \rho} + \frac{\tau_2^{\langle -2 \rangle}(x)}{(4 \om_k \rho)^2}\biggr) e_k - 
\sum_{\substack{j = 1 \\ j \ne k}}^4 \frac{a_{jk}(x)}{\om_j - \om_k} e_j 
 + \frac{\Upsilon(x, \rho)}{\rho^2}\right).
$$

Returning from the system \eqref{sysL} to equation \eqref{eqv}, we arrive at the following theorem.

\begin{thm} \label{thm:Birk}
For any $\kappa = \overline{1,8}$, $h > 0$ and some $\rho^* > 0$, equation \eqref{eqv} has a fundamental system of solutions $\{ y_k(x, \rho) \}_{k = 1}^4$ whose quasi-derivatives $y_k^{[j]}(x, \rho)$, $k = \overline{1,4}$, $j = \overline{0, 3}$, are continuous by $x \in [0, 1]$ and $\rho \in \overline{\Gamma}_{\kappa, h, \rho^*}$, analytic in $\rho \in \Gamma_{\kappa, h, \rho^*}$, for each fixed $x \in [0, 1]$, and satisfy the asymptotic relation
\begin{equation} \label{asympty}
y_k^{[j]}(x, \rho) = (\rho \om_k)^j \exp(\rho \om_k x) \left( 1 - \frac{\tau_2^{\langle -1 \rangle}(x)}{4 \rho \om_k} + \frac{\tau_2^{\langle -2 \rangle}(x)}{(4 \rho \om_k)^2} - \sum_{\substack{l = 1 \\ l \ne k}}^4 \biggl( \frac{\om_l}{\om_k}\biggr)^j \frac{a_{lk}(x)}{\rho^2 (\om_l - \om_k)} + \frac{\Upsilon(x, \rho)}{\rho^2} \right).
\end{equation}
\end{thm}

It follows from \eqref{defsig} and \eqref{A2} that
\begin{align*}
a_{lk}(0) & = -\frac{1}{4} \tau_2(0) \om_l^{-2} \om_k, \\
a_{lk}(1) & = -\frac{1}{4} \tau_2(1) \om_l^{-2} \om_k + \frac{1}{4} \sigma_1(1) \bigl( \om_l^{-3} \om_k^2 - \om_l^{-1}\bigr).
\end{align*}
Using the latter relations together with \eqref{asympty}, we conclude that, for $k = \overline{1,4}$, $j = \overline{0,3}$,
\begin{align} \label{asympty0}
    y_k^{[j]}(0, \rho) & = (\rho \om_k)^j \left( 1 + \frac{c_{(0)jk} t_0}{\rho^2} + \eps(\rho) \right), \\ \label{asympty1}
    y_k^{[j]}(1,\rho) & = (\rho \om_k)^j \left( 1 - \frac{\theta}{4 \rho \om_k} + \frac{\theta^2}{32 (\rho \om_k)^2} + \frac{c_{(0)jk} t_1}{\rho^2} + \frac{c_{(1)jk} \sigma}{\rho^2} + \eps(\rho)\right),
\end{align}
where $\eps(\rho) = \frac{\Upsilon(\rho)}{\rho^2}$ and, for brevity,
\begin{gather*}
\theta := \int_0^1 \tau_2(x) \, dx, \quad t_0 := \tau_2(0), \quad t_1 := \tau_2(1), \quad \sigma := \sigma_1(1) = \int_0^1 \tau_1(x) \, dx, \\
c_{(0)jk} := \frac{1}{4} \sum_{\substack{l = 1 \\ l \ne k}}^4 \biggl( \frac{\om_l}{\om_k}\biggr)^j \frac{\om_l^{-2} \om_k}{\om_l - \om_k}, \quad
c_{(1)jk} := -\frac{1}{4} \sum_{\substack{l = 1 \\ l \ne k}}^4 \biggl( \frac{\om_l}{\om_k}\biggr) \frac{\om_l^{-3} \om_k^2 - \om_l^{-1}}{\om_l - \om_k}.
\end{gather*}

\section{Eigenvalue asymptotics}

In this section, we obtain the eigenvalue asymptotics for the boundary value problems $\mathcal L_k$, $k = 1, 2, 3$.
The results of \cite{Bond22-asympt} imply the following proposition.

\begin{prop} \label{prop:rough}
Suppose that $\tau_2, \tau_1 \in L_2[0,1]$ and $\tau_0 \in W_2^{-1}[0,1]$. Then, for each $k = 1, 2, 3$, the spectrum of the boundary value problem $\mathcal L_k$ is a countable set of eigenvalues $\{ \la_{n,k} \}_{n \ge 1}$ (counted with multiplicities) such that
\begin{equation} \label{rough}
\la_{n,k} = (-1)^k \left( \frac{\pi}{\sin\tfrac{\pi k}{4}} \biggl(n + \chi_{k,0} + \frac{\chi_{k,1}}{n} + \frac{\varkappa_{n}}{n}\biggr) \right)^n, \quad
n \in \mathbb N,
\end{equation}
where $\chi_{k,0}$ and $\chi_{k,1}$ are some constants, $\chi_{k,0}$ does not depend on $(\tau_0, \tau_1, \tau_2)$, and $\chi_{k,1}$ depends only on $\theta = \int_0^1 \tau_2(x) \, dx$. The notation $\{ \varkappa_n \}$ is used for various sequences of $l_2$, in particular, in \eqref{rough} and analogous formulas below, $\varkappa_n$ depends on $k$.
\end{prop}

For $k = 2$, the constants in \eqref{rough} are known from the previous studies \cite{McL78, Pol23}: $\chi_{2,0} = \frac{1}{2}$, $\chi_{2,1} = -\frac{\theta}{4 \pi^2}$.
In this section, we find $\chi_{k,0}$ and $\chi_{k,1}$ for $k = 1, 3$ and improve the remainder term of \eqref{rough}, taking $\tau_2 \in W_2^1[0,1]$ into account. The main result is presented in the following theorem.

\begin{thm} \label{thm:eigen}
The following asymptotic relations hold for $n \ge 1$:
\begin{align} \nonumber 
\la_{n,2\pm 1} = - \biggl( & \Bigl( \sqrt 2 \pi n + \frac{\pi}{2 \sqrt 2}\Bigr)^4 - \theta \Bigl( \sqrt 2 \pi n + \frac{\pi}{2 \sqrt 2}\Bigr)^2 \\ \label{asymptla13} & - \frac{t_0 + t_1 \mp 4\sigma}{\sqrt 2} \Bigl( \sqrt 2 \pi n + \frac{\pi}{2 \sqrt 2}\Bigr) + n \varkappa_n \biggr), \\ \label{asymptla2}
\la_{n,2} = \Bigl( \pi n & + \frac{\pi}{2} \Bigr)^4 - \theta \Bigl( \pi n + \frac{\pi}{2} \Bigr)^2 + (t_0 + t_1) \Bigl( \pi n + \frac{\pi}{2} \Bigr) + n \varkappa_n.
\end{align}
\end{thm}

\begin{proof}
Note that $\mathcal L_1$ is adjoint to the boundary value problem for equation \eqref{eqv} with $(\tau_0, \tau_1, \tau_2)$ replaced by $(\overline{\tau_0}, -\overline{\tau_1}, \overline{\tau_2})$ subject to the boundary conditions \eqref{bc3}. Thus, the eigenvalue asymptotics \eqref{asymptla13} for $\la_{n,1}$ and for $\la_{n,3}$ can be easily obtained from one another. Therefore, it is sufficient to prove \eqref{asymptla13} for $\la_{n,3}$.

Let $C_k(x, \la)$, $k = \overline{1,4}$, be the solutions of equation \eqref{eqv} under the initial conditions
\begin{equation} \label{initC}
C_k^{[j-1]}(0, \la) = \de_{k,j}, \quad j = \overline{1,4},
\end{equation}
where $\de_{k,j}$ is the Kronecker delta.

Note that the boundary conditions \eqref{bc3} are equivalent to 
$$
y(0) = y'(0) = y^{[2]}(0) = 0, \quad y(1) = 1.
$$
Hence, the eigenvalues of $\mathcal L_3$ coincide with the zeros of the characteristic function
$$
\Delta(\la) = \begin{vmatrix}
                C_1(0, \la) & C_2(0, \la) & C_3(0, \la) & C_4(0,\la) \\
                C_1'(0, \la) & C_2'(0, \la) & C_3'(0, \la) & C_4'(0,\la) \\
                C_1^{[2]}(0, \la) & C_2^{[2]}(0, \la) & C_3^{[2]}(0, \la) & C_4^{[2]}(0,\la) \\
                C_1(1, \la) & C_2(1, \la) & C_3(1, \la) & C_4(1,\la)
            \end{vmatrix}.
$$

Let $\la = \rho^4$, $\rho \in \Gamma_{\kappa, h, \rho^*}$, and let $\{ y_k(x, \rho) \}_{k = 1}^4$ be the Birkhoff-type solutions from Theorem~\ref{thm:Birk}.
Then, the both matrix functions $C(x, \la) := [C_k^{[j-1]}(x, \la)]_{j,k = 1}^4$ and $Y(x, \rho) := [y_k^{[j-1]}(x, \rho)]_{j,k = 1}^4$ are fundamental solutions of the system \eqref{sys}. Hence
$$
C(x, \la) = Y(x, \rho) \mathcal A(\rho),
$$
where $\mathcal A(\rho) = (Y(0,\rho))^{-1}$. Consequently,
\begin{equation} \label{relDelta}
\Delta(\la) = D(\rho) \det(\mathcal A(\rho)),
\end{equation}
where
\begin{equation} \label{defD}
D(\rho) = \begin{vmatrix}
                y_1(0, \rho) & y_2(0, \rho) & y_3(0, \rho) & y_4(0,\rho) \\
                y_1'(0, \rho) & y_2'(0, \rho) & y_3'(0, \rho) & y_4'(0,\rho) \\
                y_1^{[2]}(0, \rho) & y_2^{[2]}(0, \rho) & y_3^{[2]}(0, \rho) & y_4^{[2]}(0,\rho) \\
                y_1(1, \rho) & y_2(1, \rho) & y_3(1, \rho) & y_4(1,\rho) 
          \end{vmatrix}.
\end{equation}
It follows from \eqref{asympty} that $\det(\mathcal A(\rho)) \ne 0$ for sufficiently large $|\rho|$, $\rho \in \overline{\Gamma}_{\kappa, h, \rho^*}$. Therefore, for such values of $\rho$, a complex $\la = \rho^4$ is an eigenvalue of the problem $\mathcal L_3$ if and only if $D(\rho) = 0$. 

One can easily see that the images of the two closed sectors $\overline{\Gamma}_1$ and $\overline{\Gamma}_2$ of the $\rho$-plane cover the whole $\la$-plane. Furthermore, the zeros of $D(\rho)$ for sufficiently large $|\rho|$ lie in the neighbourhood of their joint boundary $\arg \rho = \frac{\pi}{4}$. For definiteness, consider $\Gamma_1$, for which
$\om_1 = -1$, $\om_2 = i$, $\om_3 = -i$, $\om_4 = 1$.
Substituting \eqref{asympty0} and \eqref{asympty1} into \eqref{defD} for $\rho \in \Gamma_{1,h,\rho^*}$, we derive
\begin{align} \label{relD}
D(\rho) & = \rho^3 \exp(\rho \om_4) d(\rho), \\ \nonumber
d(\rho) & = r_1(\rho) - r_2(\rho) \exp(\rho(\om_3 - \om_4)) + \eps(\rho), \\ \nonumber
r_1(\rho) & = -4 i + \frac{i \theta}{\rho} + \frac{2 i \sigma}{\rho^2} - \frac{i (t_0 + t_1)}{2 \rho^2} - \frac{i \theta^2}{8\rho^2}, \\ \nonumber
r_2(\rho) & = 4 - \frac{i\theta}{\rho} + \frac{2\sigma}{\rho^2} - \frac{t_0 + t_1}{2\rho^2} - \frac{\theta^2}{8 \rho^2}.
\end{align}

Consequently, the equation $D(\rho) = 0$ in $\Gamma_{1,h,\rho^*}$ can be reduced to the form
$$
\exp(\rho(\om_4 - \om_3)) = \frac{r_2(\rho)}{r_1(\rho)} + \eps(\rho) =
i + \frac{(1 + i) \theta}{4 \rho} + \frac{i \sigma}{\rho^2} - \frac{i(t_0 + t_1)}{4 \rho^2} + \frac{\theta^2}{16 \rho^2} + \eps(\rho).
$$

Taking the logarithm and using the standard approach, based on Rouche's theorem (see \cite{Bond22-asympt}), we find the roots
$$
(1 + i) \rho_n = 2 \pi i n + \frac{\pi i}{2} + \frac{(1 + i) \theta}{4 i \rho_n} + \frac{\sigma}{\rho_n^2} - \frac{t_0 + t_1}{4 \rho_n^2} + \eps(\rho_n)
$$
for $n \ge n_0$, where $n_0$ is sufficiently large.
Recall that $\eps(\rho) = \frac{\Upsilon(\rho)}{\rho^2}$, where $\Upsilon(\rho)$ satisfies \eqref{Ups1} and \eqref{Ups2}. Therefore, $\eps(\rho_n) = \frac{\varkappa_n}{n^2}$, $\{ \varkappa_n \} \in l_2$. Finally, we arrive at the relation
\begin{equation} \label{asymptrhon}
\rho_n = \exp\biggl( \frac{\pi i}{4}\biggr) \left( \sqrt 2 \pi n + \frac{\pi}{2 \sqrt 2} - \frac{\theta}{4 (\sqrt 2 \pi n + \frac{\pi}{2 \sqrt 2})} - \frac{\sigma}{2 \sqrt 2 (\pi n)^2} + \frac{t_0 + t_1}{8 \sqrt 2 (\pi n)^2} + \frac{\varkappa_n}{n^2}\right), \quad n \ge n_0.
\end{equation}

Computing $\la_n = \rho_n^4$, we arrive at the asymptotics \eqref{asymptla13} for sufficiently large values of $n$. It remains to show that the numbering in \eqref{asymptla13} starts from $n = 1$. By virtue of Theorem~1.1 in \cite{Bond22-asympt}, it is sufficient to prove this for the eigenvalues $\{ \la_{n,3}^0 \}$ of the boundary value problem $\mathcal L_3^0$ of the same form as $\mathcal L_3$ but with the zero coefficients $\tau_j^0 = 0$, $j = 0, 1, 2$. Using \eqref{relDelta}, \eqref{defD}, and the Birkhoff solutions $y_k^0(x, \rho) = \exp(\rho \om_k x)$, $k = \overline{1,4}$, of the equation $y^{(4)} = \rho^4 y$, we obtain the relation
$$
\Delta^0(\la) = \frac{\sinh(\rho) - \sin(\rho)}{2 \rho^3}
$$
for the characteristic function of $\mathcal L_3^0$.

Using the Taylor series
$$
\Delta^0(\rho^4) = \frac{1}{3!} + \frac{\rho^4}{7!} + \frac{\rho^8}{11!} + \dots, 
$$
we show that $\Delta^0(\rho^4) \ne 0$ for $|\rho| \le 5$. It is sufficient to consider $0 \le \arg \rho \le \frac{\pi}{2}$. For $0 \le \arg \rho \le \frac{\pi}{8}$, $|\rho| \ge 5$, we have
\begin{align*}
2 |\sinh(\rho)| & \ge |\exp(\rho)| - |\exp(-\rho)| \ge \exp(5 \cos \tfrac{\pi}{8}) - \exp(-5) > 100, \\
2 |\sin(\rho)| & \le |\exp(i \rho)| + |\exp(-i\rho)| \le \exp(5 \sin \tfrac{\pi}{8}) + 1 < 10,
\end{align*}
so $|\sinh(\rho)| > |\sin(\rho)|$ and $\Delta^0(\rho^4)$ has no zeros in this sector. Analogously $\Delta^0(\rho^4) \ne 0$ for $\frac{3\pi}{8} \le \arg \rho \le \frac{\pi}{2}$. Hence, all the zeros of $\Delta^0(\rho^4)$ lie in the region 
$$
\mathcal G := \left\{ \rho \colon |\rho| > 5, \tfrac{\pi}{8} < \arg \rho < \tfrac{3\pi}{8} \right\}.
$$

Consider the functions
\begin{gather*}
f(\rho) := 2 \sinh(\rho) - 2 \sin(\rho) = g(\rho) + s(\rho), \\
g(\rho) := \exp(\rho) - i \exp(-i\rho), \quad s(\rho) := i \exp(i \rho) - \exp(\rho).
\end{gather*}
Obviously, in $\mathcal G$, the function $g(\rho)$ has the zeros 
$$
\rho_n^{\diamond} = \exp\left( \tfrac{\pi i}{4}\right) \left( \sqrt 2 \pi n + \frac{\pi}{2 \sqrt 2}\right), \quad n \ge 1,
$$
which coincide with the main term of \eqref{asymptrhon},
and 
$$
|s(\rho)| \le 2 \exp(-5 \sin \tfrac{\pi}{8}) < 0.3, \quad \rho \in \mathcal G.
$$

Define the region
$$
\mathcal G_{\de} := \{ \rho \in \mathcal G \colon |\rho - \rho_k^{\diamond}| \ge \de, \, k \ge 1 \}, \quad \de := 0.1.
$$
For $\rho = \rho_n^{\diamond} + z \in G_{\de}$, $|z| \ge \de$,
we obtain the estimate
$$
|g(\rho)| \ge |\exp(\rho_n^{\diamond} + z)| |1 - \exp(-(1+i) z)| \ge \exp(5 \cos \tfrac{3\pi}{8}) \frac{|z|}{\sqrt{2}} > 0.4.
$$

Therefore, we get $\dfrac{|s(\rho)|}{|g(\rho)|} < 1$ in $\mathcal G_{\de}$. Applying Rouch\'e's theorem to $f(\rho)$ and $g(\rho)$, we conclude that $f(\rho)$ has only the simple zeros $\{ \rho_n^0 \}_{n \ge 1}$ in $\mathcal G$ such that $|\rho_n^0 - \rho_n^{\diamond}| < 0.1$, and, obviously, so does $\Delta^0(\rho^4)$. This concludes the proof of \eqref{asymptla13}.

The asymptotics \eqref{asymptla2} is proved analogously.
\end{proof}

\begin{remark}
The main part of the asymptotics \eqref{asymptla2} coincides with the one obtained by Polyakov \cite{Pol23} for the case $\tau_j \in W_1^{j + 1}[0,1]$, $\tau_j(0) = \tau_j(1)$ for $j = 0, 2$, $\tau_1 \equiv 0$. Sharp eigenvalue asymptotics for the fourth-order boundary value problems with the boundary conditions \eqref{bc1} and \eqref{bc3}, to the best of the author's knowledge, have not been investigated before, so the asymptotic formula \eqref{asymptla13} is novel even for the case of regular coefficients.
\end{remark}

\section{Weight numbers}

In this section, we recall the definition of the weight numbers from \cite{Bond23-loc} and derive for them precise asymptotics, which play an important role in the spectral data characterization (see \cite{Bond23-nsc}).

Define the functions
$$
\Delta_k(\la) := \det([C_j^{[4-s]}(1,\la)]_{s, j = k+1}^4), \quad
\Delta_k^+(\la) := \det([C_j^{[4-s]}(1,\la)]_{s=\overline{k+1,4}, j = k, \overline{k+2,4}}), \quad k = 1, 2, 3,
$$
where $C_j(x, \la)$, $j = \overline{1,4}$, are the solutions of equation \eqref{eqv} satisfying the initial conditions \eqref{initC}. One can show in the standard way that, for each $k \in \{ 1, 2, 3\}$, the zeros of $\Delta_k(\la)$ coincide with the eigenvalues $\{ \la_{n,k} \}_{n \ge 1}$ of the corresponding boundary value problem $\mathcal L_k$.
The weight numbers are defined as follows:
$$
\be_{n,k} := -\Res_{\la = \la_{n,k}} \frac{\Delta_k^+(\la)}{\Delta_k(\la)}, \quad n \ge 1, \quad k = 1, 2, 3.
$$

According to the asymptotics of Theorem~\ref{thm:eigen}, the eigenvalues $\la_{n,k}$ are simple for $n \ge n_0$ with a sufficiently large $n_0$. Hence, we have
\begin{equation} \label{defbe}
\be_{n,k} = -\frac{\Delta^+_k(\la_{n,k})}{\frac{d}{d\la}\Delta_k(\la_{n,k})}, \quad n \ge n_0, \quad k = 1, 2, 3.
\end{equation}

The main result of this section is provided in the following theorem.

\begin{thm} \label{thm:weight}
For $n \ge n_0$, the weight numbers satisfy the asymptotic relations
\begin{align} \label{asymptbe13}
\be_{n,2\pm 1} & = -4 \la_{n,2\pm 1} \left( 1 + \frac{t_0 + \theta}{ 8(\pi n)^2}  + \frac{\varkappa_n}{n^2}\right), \\ \label{asymptbe2}
\be_{n,2} & = -4 \la_{n,2} \left( 1 + \frac{t_0 + 2 \theta}{4 (\pi n)^2} + \frac{\varkappa_n}{n^2}\right).
\end{align}
\end{thm}

\begin{proof}
Similarly to Theorem~\ref{thm:eigen}, let us focus on the proof of \eqref{asymptbe13} for $\be_{n,3}$. The proofs for $\be_{n,1}$ and $\be_{n,2}$ are analogous. 

Consider $\rho \in \Gamma_{1,h,\rho^*}$, $\la = \rho^4$.
For brevity, put $\Delta(\la) := \Delta_3(\la)$ and
$$
\Delta^+(\la) := \Delta_3^+(\la) = -
\begin{vmatrix}
C_1(0, \la) & C_2(0, \la) & C_3(0,\la) & C_4(0,\la) \\
C_1'(0, \la) & C_2'(0, \la) & C_3'(0,\la) & C_4'(0,\la) \\
C_1^{[3]}(0, \la) & C_2^{[3]}(0, \la) & C_3^{[3]}(0,\la) & C_4^{[3]}(0,\la) \\
C_1(1, \la) & C_2(1,\la) & C_3(1,\la) & C_4(1,\la)
\end{vmatrix}.
$$

Analogously to \eqref{relDelta}, we obtain the relation 
\begin{equation} \label{relDelta+}
\Delta^+(\la) = -D^+(\rho) \det(\mathcal A(\rho)),
\end{equation}
where
$$
D^+(\rho) = \begin{vmatrix}
y_1(0, \la) & y_2(0, \la) & y_3(0,\la) & y_4(0,\la) \\
y_1'(0, \la) & y_2'(0, \la) & y_3'(0,\la) & y_4'(0,\la) \\
y_1^{[3]}(0, \la) & y_2^{[3]}(0, \la) & y_3^{[3]}(0,\la) & y_4^{[3]}(0,\la) \\
y_1(1, \la) & y_2(1,\la) & y_3(1,\la) & y_4(1,\la)
\end{vmatrix}.
$$
Similarly to \eqref{relD}, we derive
\begin{align} \label{relD+}
D^+(\rho) & = \rho^4 \exp(\rho \om_4) d^+(\rho), \\ \nonumber
d^+(\rho) & = r_1^+(\rho) - r_2^+(\rho) \exp(\rho(\om_3 - \om_4)) + \eps(\rho), \\ \nonumber
r_1^+(\rho) & = 4i - \frac{i \theta}{\rho} - \frac{2 i \sigma}{\rho^2} - \frac{i (t_0 - t_1)}{2 \rho^2} + \frac{i \theta^2}{8 \rho^2}, \\ \nonumber
r_2^+(\rho) & = 4 i + \frac{\theta}{\rho} + \frac{2 i \sigma}{\rho^2} + \frac{i (t_0 - t_1)}{2 \rho^2} - \frac{i \theta^2}{8 \rho^2}.
\end{align}

Substituting \eqref{relDelta}, \eqref{relD}, \eqref{relDelta+}, and \eqref{relD+} into \eqref{defbe} for $k = 3$, we obtain 
\begin{equation} \label{relbe}
\be_{n,3} = 4 \la_{n,3} \frac{d^+(\rho_n)}{\frac{d}{d\rho} d(\rho_n)}, 
\end{equation}
where $\la_{n,3} = \rho_n^4$ and $\rho_n$ satisfies \eqref{asymptrhon}.
Calculations show that 
\begin{align} \nonumber
\frac{d^+(\rho)}{\frac{d}{d\rho} d(\rho)} & = \frac{r_1^+(\rho) r_2(\rho) - r_2^+(\rho) r_1(\rho)}{\frac{d}{d\rho} r_1(\rho) r_2(\rho) - \frac{d}{d\rho} r_2(\rho) r_1(\rho) - (\om_3 - \om_4) r_1(\rho) r_2(\rho)} + \eps(\rho) \\ \label{calcbe}
& = - \left( 1 + \frac{i (t_0 + \theta)}{4 \rho^2} + \eps(\rho)\right).
\end{align}

Combining \eqref{asymptrhon}, \eqref{relbe}, and \eqref{calcbe}, we arrive at \eqref{asymptbe2} for $\be_{n,3}$.
\end{proof}

\medskip

{\bf Acknowledgment.} This work was supported by Grant 21-71-10001 of the Russian Science Foundation, https://rscf.ru/en/project/21-71-10001/

\medskip

\noindent Natalia Pavlovna Bondarenko \\

\noindent Department of Mechanics and Mathematics, Saratov State University, \\
Astrakhanskaya 83, Saratov 410012, Russia, \\
e-mail: {\it bondarenkonp@info.sgu.ru}

\end{document}